\documentclass[11pt,letterpaper]{article}

\usepackage[margin=1in]{geometry}

\usepackage{theorem,latexsym,graphicx}
\usepackage{amsmath,amssymb,enumerate}
\usepackage{xspace} 
\usepackage{bm} %
\usepackage{ifpdf}
\usepackage{shadow,shadethm,color}
\usepackage[compact]{titlesec}
\usepackage{algorithm, algorithmic}
\usepackage{subfigure}
\usepackage{verbatim}

\definecolor{Darkblue}{rgb}{0,0,0.4}
\definecolor{Brown}{cmyk}{0,0.81,1.,0.60}
\definecolor{Purple}{cmyk}{0.45,0.86,0,0}
\newcommand{\mydriver}{hypertex}
\ifpdf
 \renewcommand{\mydriver}{pdftex}
\fi
\usepackage[breaklinks,\mydriver]{hyperref}
\hypersetup{colorlinks=true,
            citebordercolor={.6 .6 .6},linkbordercolor={.6 .6 .6},%
citecolor=Darkblue,urlcolor=black,linkcolor=Darkblue,pagecolor=black}
\newcommand{\lref}[2][]{\hyperref[#2]{#1~\ref*{#2}}}


\newtheorem{theorem}{Theorem}[section]
\newtheorem{definition}[theorem]{Definition}

\newtheorem{lemma}[theorem]{Lemma}
\newshadetheorem{lemmashaded}[theorem]{Lemma}

\newtheorem{corollary}[theorem]{Corollary}
\newenvironment{proof}{

\noindent{\bf Proof:}}
{\hfill$\blacksquare$

}

\newcommand{\junk}[1]{}
\newcommand{\ignore}[1]{}

\newcommand{\R}[0]{{\ensuremath{\mathbb{R}}}}

\newcommand{\Z}[0]{{\ensuremath{\mathbb{Z}}}}

\def\ceil#1{\left\lceil #1 \right\rceil}

\newcommand{\pr}[1]{{\rm Pr} \left[ #1 \right]}

\newcommand{\E}{\mathbb{E}}
\newcommand{\Eb}[1]{\E\left[#1\right]} 
\newcommand{\var}[1]{\mathop{Var}\left[#1\right]}

\newcounter{note}[section]

\newcommand{\qedsymb}{\hfill{\rule{2mm}{2mm}}}
\renewenvironment{proof}{\begin{trivlist} \item[\hspace{\labelsep}{\bf
\noindent Proof.\/}] }{\qedsymb\end{trivlist}}%

\newcommand{\initOneLiners}{%
    \setlength{\itemsep}{0pt}
    \setlength{\parsep }{0pt}
    \setlength{\topsep }{0pt}
}

\newcommand{\ve}[1]{{\cal V}\left(#1\right)} 

\begin{document}

\title{Bernstein-like Concentration and Moment Inequalities for Polynomials of Independent Random Variables:  Multilinear Case}

\author{Warren Schudy\thanks{IBM T. J. Watson Research Center, P.O.
Box 218, Yorktown Heights, NY
     10598. {\tt wjschudy@us.ibm.com}} \and Maxim Sviridenko\thanks{University of Warwick,  {\tt sviri@dcs.warwick.ac.uk}}}
\date{}

\maketitle

\begin{abstract}

Polynomials of independent random variables arise in a variety of fields such as Machine Learning, Analysis of Boolean Functions, Additive Combinatorics, Random Graphs Theory, Stochastic Partial  Differential Equations etc. They naturally model the expected value of objective function (or lefthand side of constraints) for randomized rounding algorithms for non-linear optimization problems where one finds a solution of an "easy" continuous problem and rounds it to a solution of a "hard" integral problem (one such example is Convex Integer Programming \cite{BCL}).  To measure the performance guarantee of such algorithms one needs analogously to the analysis employed by Raghavan and Thompson \cite{RT}  for boolean integer programming problems  an analog of Chernoff Bounds for polynomials of independent random variables. There are many known forms and variations of Chernoff Bounds. One of the tightest ones is based on a variance of a sum of random variables known as Bernstein inequality. Another popular albeit a weaker version is using an estimate of a variance through the expectation. The later versions of concentration inequalities for polynomials of independent random variables are known \cite{KV,SS}. In this paper we derive an analog of Bernstein Inequality for multilinear polynomials of independent random variables.

We show that the probability that a multilinear polynomial $f$ of independent random variables exceeds its mean by $\lambda$ is at most $e^{-\lambda^2 / (R^q Var(f))}$ for sufficiently small $\lambda$, where $R$ is an absolute constant.  This matches (up to constants in the exponent) what one would expect from the central limit theorem. Our methods handle a variety of types of random variables including Gaussian, Boolean, exponential, and Poisson. Previous work by Kim-Vu and Schudy-Sviridenko gave bounds of the same form that involved less natural parameters in place of the variance.
\end{abstract}

\section{Introduction}

Polynomials of independent random variables arise in a variety of fields such as Machine Learning, Analysis of Boolean Functions, Additive Combinatorics, Random Graphs Theory, Stochastic Partial  Differential Equations etc. They naturally model the expected value of objective function (or lefthand side of constraints) for randomized rounding algorithms for non-linear optimization problems where one finds a solution of an "easy" continuous problem and rounds it to a solution of a "hard" integral problem (one such example is Convex Integer Programming \cite{BCL}).  To measure the performance guarantee of such algorithms one needs analogously to the analysis employed by Raghavan and Thompson \cite{RT}  for boolean integer programming problems  an analog of Chernoff Bounds for polynomials of independent random variables. There are many known forms and variations of Chernoff Bounds. One of the tightest ones is based on a variance of a sum of random variables known as Bernstein Inequality \cite{B,B1,U}. Another popular albeit a weaker version is using an estimate of a variance through the expectation. The later versions of concentration inequalities for polynomials of independent random variables are known \cite{KV,SS}. In this paper we derive an analog of Bernstein Inequality for multilinear polynomials of independent random variables.

Perhaps the most celebrated theorem in statistics is the central limit theorem. This theorem (actually family of theorems) states conditions under which a sum of $n$ independent random variables converges to being normally (i.e.\ Gaussian) distributed as $n \to \infty$. Let $Y_1,Y_2,\dots$ be a sequence of independent random variables. Let $Var[Z]=\Eb{(Z - \Eb{Z})^2}=\Eb{Z^2}-\Eb{Z}^2$ be the variance of the random variable $Z$. Various central limit theorems state various conditions on the $Y_i$ under which
\begin{align}
\lim_{n \to \infty} \pr{\sum_{i=1}^n (Y_i - \Eb{Y_i}) \ge \lambda \sqrt{\var{\sum_{i=1}^n Y_i}}} &= \int_{-\infty}^\lambda \frac{1}{\sqrt{2\pi}} e^{-x^2/2} dx \label{eqn:clt}
\end{align}
for any $\lambda \in \R$.
One sufficient condition is that the $Y_i$ are identically distributed with finite variance \cite{jacod}. Another set of sufficient conditions is that there exists $M, \epsilon > 0$ such that all $\Eb{|Y_i|^{2+\epsilon}} \le M$ and $\lim_{n \to \infty}Var[\sum_{i=1}^n Y_i]=\infty$ \cite{jacod}.

The rate of convergence of the limit is often of interest. The Berry-Esseen theorem \cite{jacod} states that when the $Y_i$ are identically distributed with finite $\Eb{|Y_1|^3}$ and $\Eb{Y_1^2}=\sigma^2$ then
\begin{align}
\left|\pr{\sum_{i=1}^n (Y_i - \Eb{Y_i}) \ge \lambda \sqrt{\var{\sum_{i=1}^n Y_i}}} - \int_{-\infty}^\lambda \frac{1}{\sqrt{2\pi}} e^{-x^2/2} dx\right| \le \frac{0.77\Eb{|Y_1|^3}}{\sigma^3 \sqrt{n}}
.\label{eqn:berryEsseen}
\end{align}

Many applications require upper bounds on the probability of large deviations for finite $n$. The Berry-Esseen bound (\ref{eqn:berryEsseen}) is exponentially far from tight for many such applications, for example the probability that at least three-quarters of a sequence of coin flips are heads is $2^{-\Theta(n)}$ but the Berry-Esseen bound is $O(1/\sqrt{n})$. Fortunately it is possible to do much better in many cases. For example if the $Y_i$ are independent random variables with $0 \le Y_i \le 1$ then a standard Bernstein inequality (e.g. Theorem 2.3 (b) in \cite{mcdiarmid98}) states that
\begin{align} 
\pr{\sum_{i=1}^n (Y_i - \Eb{Y_i}) \ge \lambda} &\le \exp\left(-\frac{\lambda^2}{2\mu + 2\lambda/3}\right) \label{eq:chernoff}
\end{align}
for any $\lambda>0$ where $\mu=\Eb{\sum_{i=1}^n Y_i}$. Note that the small-$\lambda$ probability bound is roughly $\exp\left(-\frac{\lambda^2}{2\mu}\right)$, which matches the Gaussian behavior suggested by the central limit theorem except for the use of the upper-bound for the variance $\mu$ in place of the variance. This discrepancy can be remedied, yielding another variant of the Bernstein inequality (see Theorem 2.7 in \cite{mcdiarmid98})
\begin{align} 
\pr{\sum_{i=1}^n (Y_i - \Eb{Y_i}) \ge \lambda} &\le \exp\left(-\frac{\lambda^2}{2V + 2\lambda/3}\right) \label{eq:chernoffVar}
\end{align}
where $V=Var(\sum_{i=1}^n Y_i)$.
For $\lambda \le V$ this matches (up to constants in the exponent) what the central limit theorem suggests.

Kim and Vu introduced variants of Chernoff bound (\ref{eq:chernoff}) for polynomials of independent Boolean random variables \cite{KV}.
Vu \cite{V1} tightened and generalized the bounds to handle independent random variables with arbitrary distributions in the interval $[0,1]$.
Schudy and Sviridenko \cite{SS} proved a  stronger  concentration inequality for polynomials of independent random variables  satisfying a general condition (see Definition \ref{var2}). Note that \cite{V1} contains one extension not handled in \cite{SS} and this paper, namely using less then $q$ (the degree of the polynomial) smoothness parameters.
 These bounds share the Gaussian-like behavior for small $\lambda$ with (\ref{eq:chernoff}), but they use an upper bound on the variance that is more complicated than the $\mu$ used in (\ref{eq:chernoff}). The behavior for large $\lambda$ is also different. Our main contribution is an analog of (\ref{eq:chernoffVar}) for polynomial $f(Y)$ of power $q$:
\begin{align}
\pr{|f(Y) - \Eb{f(Y)}| \ge \lambda } &\le e^2 \exp \left(-\frac{\lambda^2 }{Var[f(Y)]R^q}\right) \label{eq:us}
\end{align}
for all sufficiently small $\lambda$ (see Theorem \ref{main} for the precise statement), where $R$ is an absolute constant. What values of $\lambda$ are ``sufficiently small'' depends on parameters $\mu_1,\mu_2,\dots,\mu_q$ defined in the next section. For example in the setting of (\ref{eq:chernoffVar}) we reproduce that bound up to constants in the exponent: Gaussian-like tails for $\lambda \le V$ and exponential tails for larger $\lambda$. Some polynomials require $\lambda$ to be so small that $e^2 e^{-\lambda^2 / R^q}$ always exceeds 1 and hence (\ref{eq:us}) is vacuous. We expect that most applications will involve $\lambda$ sufficiently small for (\ref{eq:us}) to apply.

The improvement of (\ref{eq:us}) compared to the concentration inequalities in Schudy-Sviridenko \cite{SS} is analogous to the improvement of (\ref{eq:chernoffVar}) compared to (\ref{eq:chernoff}).
There are countless applications of Bernstein Inequality (or its variants known as Chernoff or Hoeffding Bounds) and its martingale versions \cite{B1,FR}. Recent algorithmic applications of the martingale version of Bernstein Inequality that require dependence on variance instead of expectation are \cite{DHK} and \cite{MSS}. Analogously, we expect that in the future there will be many applications (e.g. counting in random graphs) where one would necessarily need a stronger inequality of Theorem \ref{main} (analog of (\ref{eq:chernoffVar}) for polynomials)  instead of Theorem \ref{main1specialSS} (analog of (\ref{eq:chernoff}) for polynomials). Note that before our work such statements were not even known for boolean random variables.

\subsection{Our Results}
We are given a hypergraph $H=({\cal V}(H),{\cal H}(H))$ consisting of a set ${\cal V}(H)=\{1,2,\ldots,n\}=[n]$ of vertices and a set ${\cal H}(H)$ of hyperedges. A hyperedge $h$ consists of a set ${\cal V}(h) \subseteq {\cal V}(H)$ of $|{\cal V}(h)| \le q$ vertices. We are also given a  weight $w_h$ for each $h \in {\cal H}(H)$. For each such weighted hypergraph and real-valued weight $w_h$ for its hyperedges, we define a multilinear polynomial
\begin{align}
f(x)&= \sum_{h\in {\cal  H}(H) } w_h\prod_{v\in {\cal V}(h)}x_v. \label{eqn:fmulti}
\end{align}
We call the maximum hyperedge cardinality $q$ the \emph{power} of the polynomial $f$.

We use essentially the same smoothness parameters as Kim and Vu \cite{KV,V1} in our previous work \cite{SS}.
For a given collection of independent random variables $Y=(Y_1,\dots,Y_n)$, hypergraph $H$, weights   $ w$ and integer parameter $r\ge 0$, we define
$$\mu_r=\mu_r(Y,H,w)= \max_{S\subseteq [n]:|S|=r}\left(\sum_{h\in {\cal  H}| {\cal V}(h) \supseteq S} |w_h| \prod_{v\in {\cal V}(h)\setminus S}\Eb{|Y_v|}\right).$$
 Note that $S$ need not be avertex set of some hyperedge of $H$ and may even be the empty set.
 Sometimes we will also use the notation $\mu_r(f)=\mu_r(Y,H,w)$ to emphasize the dependence on polynomial $f$.

In the previous work \cite{SS}, we proved moment and concentration  inequalities that could be viewed as an extension of (\ref{eq:chernoff}) to polynomials of random variables satisfying the following condition.
\begin{definition}\label{var2} \cite{SS}
   A random variable $Z$ is called \emph{\bf moment bounded} with real parameter $L>0$, if for any integer $i\ge 1$ we have
 $$ \Eb{\left| Z\right|^i} \le   i\cdot L\cdot\Eb{ \left|Z\right|^{i-1}}.$$
\end{definition}
That work \cite{SS} showed that three large classes of random variables are moment bounded: bounded, continuous log-concave \cite{BB,An} and discrete log-concave \cite{An}.
The results of the current paper apply to a related type of random variable.
\begin{definition}\label{var}
 A random variable $Z$ is called \emph{\bf central moment bounded} with real parameter $L>0$, if for any integer $i\ge 1$ we have
 $$ \Eb{\left| Z-\Eb{ Z} \right|^i} \le   i\cdot L\cdot\Eb{ \left| Z -\Eb{ Z } \right|^{i-1}}.$$
\end{definition}
In Section \ref{sec:exampleMom} we show that the three classes of random variables that are known to be moment bounded (i.e.\ bounded, continuous log-concave and discrete log-concave) are also central moment bounded. For example Poisson, geometric, normal (i.e.\ Gaussian), and exponential distributions are all central moment bounded.

We prove the following:
\begin{theorem}\label{main}
We are given $n$ independent central moment bounded random variables $Y=(Y_1,\dots, Y_n)$ with the same parameter $L$. We are given a multilinear polynomial $f(y)$  of  power $q$.  Let $f(Y)=f(Y_1,\dots,Y_n)$ then
\begin{equation}\label{main:eq}
Pr\left[|f(Y)-\E[f(Y)]|\ge \lambda\right]\le e^2\cdot \max\left\{ e^{-\frac{\lambda^2}{Var[f(Y)]\cdot R^q }},\max_{r\in [q]}e^{-\left(\frac{\lambda}{\mu_r L^r R^q}\right)^{1/r}}\right\},
\end{equation}
where $R$ is some absolute constant.
\end{theorem}

Quite often in the applications $\mu_r$ for $r=1,\dots,q-1$ are negligibly small and $\mu_q=1$ (e.g. \cite{V2}). In this case, the right hand side of (\ref{main:eq}) becomes
$$e^2\cdot \max\left\{ e^{-\frac{\lambda^2}{Var[f(Y)]\cdot R^q }},e^{-\frac{\lambda^{1/q}}{ L\cdot R}}\right\}$$
and our Theorem \ref{main} implies that tails of multilinear polynomials of central moment bounded random variables have Gaussian-like  distribution for $\lambda<Var[f(Y)]^{q/(2q-1)}$ and constants $L,R$.
Previous work \cite{SS} proved a similar theorem:
\begin{theorem}\label{main1specialSS} \cite{SS}
We are given $n$ independent moment bounded random variables $Y=(Y_1,\dots, Y_n)$ with the same parameter $L$. We are given a multilinear polynomial $f(x)$ with nonnegative coefficients of total power $q$.  Let $f(Y)=f(Y_1,\dots,Y_n)$  then
$$Pr\left[|f(Y)-\E[f(Y)]|\ge \lambda\right]\le e^2\cdot \max\left\{ e^{-\frac{\lambda^2}{\max_{r\in [q]}(\mu_0\mu_r\cdot L^r\cdot R^q) }},\max_{r\in [q]}e^{-\left(\frac{\lambda}{\mu_r L^r R^q}\right)^{1/r}}\right\},$$
where $R$ is some absolute constant.
\end{theorem}

We show that the parameter $Var[f(Y)]$ in Theorem \ref{main} is always at least as good as the $\max_{r\in [q]}\left(\mu_0\mu_r L^r\right)$ in Theorem \ref{main1specialSS}:
\begin{lemma}\label{lem:varBound}
For a multilinear polynomial $f$ as in Theorem \ref{main} we have
\[
Var[f(Y)] \le 2q4^q\max_{r\in[q]} \left(\mu_0(f,Y) \mu_r(f,Y) 4^rL^r\right).
\]
\end{lemma}
Lemma \ref{lem:varBound} implies that our Theorem \ref{main} dominates Theorem \ref{main1specialSS} from \cite{SS} in the common case when   the central moment boundedness parameter $L$ is of the same order as the moment boundedness parameter.

Previous work \cite{SS} showed that Theorem \ref{main1specialSS} has a tight dependence on the parameters $\mu_1,\dots,\mu_q$ up to factors of logarithms and $q^{O(q)}$ in the exponent. That lower bound only applies to bounds that depend only on those parameters and hence Theorem \ref{main}, which additionally depends on $Var[f(Y)]$, does not contradict it.

\subsection{Comparing with Hypercontractivity concentration inequality and other results}

It is well known that considering sums of centered (i.e. $\Eb{Y}=0$) and subgaussian (i.e. $\Eb{|Y|^k}\le L^kk^{k/2}\Eb{|X|}$) random variables improves the concentration bounds.  Namely the concentration arounds its mean stays gaussian even for large values of $\lambda$ unlike the case of the sum of non-centered (even boolean) random variables where the concentration bounds start to behave like the ones of exponential random variable.
Therefore, we can expect a similar phenomenon for the polynomials of independent centered subgaussian random variables, i.e. the concentration bounds for polynomials of independent centered subgaussians should have tighter concentration around the mean for larger values of $\lambda$.

Two specific examples of such variables are centered Gaussian and Rademacher ($+1$ or $-1$ with probability $1/2$) random variables. There are two concentration inequalities known in the literature specific for that setting
\begin{theorem}[Hypercontractivity Concentration Inequality]\label{hyper}
Consider a multilinear degree $q$ polynomial $f(Y)=f(Y_1,\dots,Y_n)$ of independent Normal or Rademacher random variables $Y_1,\dots,Y_n$.  Then
$$Pr[|f(Y)-\E[f(Y)]|\ge \lambda] \le e^2\cdot e^{-\left(\frac{\lambda^2}{R\cdot Var[f(Y)]}\right)^{1/q}},$$
where $Var[f(Y)]$ is the variance of the random variable $f(Y)$ and $R>0$ is   an absolute constant.
\end{theorem}
The history of these concentration and corresponding moment inequalities is quite rich see S. Janson \cite{J} (Sections V and VI).  Latala \cite{L} tightened these inequalities for Normal random variables using smoothness  parameters similar but incomparable to ours (see the next Section).

Unfortunately, the Hypercontractivity  and even Latala Concentration Inequalities do not strictly dominate our concentration inequality (Theorem \ref{main}). Our concentration behaves better for small values of $\lambda$ with respect to Hypercontractivity Concentration Inequality and for some polynomials we beat the Latala bounds  for large values of $\lambda$ since our smoothness parameters are incomparable.

The conclusion is that it is likely that there exists a yet to be discovered concentration inequality for polynomials of independent centered subgaussian random variables that dominates ours (Theorem \ref{main}), Hypercontractivity (Theorem \ref{hyper}) and Latala's \cite{L} concentration inequalities in this setting. Deriving such an inequality is a challenging open problem.

In our previous work \cite{SS}, we provide an extensive comparison of Theorem \ref{main1specialSS} and its analog for general polynomials with various known concentration inequalities for polynomials.
Mossel, O'Donnell and Oleszkiewicz \cite{MOO} showed that the distribution of a multilinear polynomial of indepedent random variables is approximately invariant with respect to the distribution of the random variables as long as the random variables have mean 0 and variance 1. In particular they bound
\begin{align*}
\left|\pr{ f(X_1,\dots,X_n) \ge \lambda } - \pr{f(G_1,\dots,G_n) \ge \lambda}\right|
\end{align*}
where $f$ is a multilinear polynomial, $X_1,\dots,X_n,G_1,\dots,G_n$ are independent random variables with mean 0 and variance 1, and $G_1,\dots,G_n$ have a Gaussian distribution. Such bounds can be considered to be a generalization of the Berry-Esseen type bounds because in the linear case the sum of Gaussians $f(G_1,\dots,G_n)$ has a Gaussian distribution. Note that as usual central limit theorem or invariance principle type of results have very wide range of applicable random variables but weaker concentration bounds (polynomial instead of exponential).

\subsection{Our Techniques}
Our work follows the same general scheme of the moment computation method developed in the proof of Theorem \ref{main1specialSS} in \cite{SS} but there are many subtle differences in the proofs since we basically want to replace each term $\mu_0\mu_q$ in the proof of the {\bf Initial Moment Lemma} from \cite{SS} with variance. For example, our Section \ref{sec:eachGraph} and the analogous Section 2.2 in \cite{SS} are devoted to bounding a certain sum (the sum over $\pi$ in (\ref{inequality3})).
Previous work \cite{SS} gets a factor $\mu_t$ for some $0 \le t \le q$ for each hyperedge in a certain hypergraph $G'$. Each connected component of $G'$ includes two hyperedge weights, call them $w_1$ and $w_2$, which contribute to a $\mu_0$ and a $\mu_q$ factor respectively in \cite{SS}. Instead of having $w_1 w_2$ contribute to a $\mu_0 \mu_q$ we do the following. We bound $w_1w_2 \le (w_1^2 + w_2^2)/2$ and then $w_1^2$ and $w_2^2$ each contribute to a factor of variance.  Our Ordering Lemma in Section \ref{technical} is different from the analogous statement in \cite{SS}. To transition from Initial Moment Lemma to the General Even Moment Lemma we use certain orthogonality properties of multi-linear polynomials which do not seem to hold for general polynomials. Our key property of random variables (central moment boundness) is different from the moment boundness in \cite{SS} which forced us to re-prove that the classical classes of discrete and continuous random variables satisfy that property.

While we were able to extend the Theorem \ref{main1specialSS} in \cite{SS} to the case of general polynomials we were not able to prove a similar extension of Theorem \ref{main}. While we believe that such a statement is true, it seems it would require another property of random variables different from moment boundness or central moment boundness.

\subsection{Outline}

The high-level organization of our analysis follows \cite{SS}. Sections \ref{sec:centered} and \ref{sec:intermediate} state and prove key lemmas on the moments of polynomials of variables with zero expectation. Section \ref{technical} proves various technical lemmas that are omitted from the main flow. Section \ref{ProofMainGen} states and proves bounds on the moments of polynomials with arbitrary expectation. Section \ref{ProofMain} uses those bounds and Markov's inequality to prove Theorem \ref{main}.  Section \ref{sec:exampleMom} shows that a wide variety of classical random variables are central moment bounded.

\section{Moment Lemma for Centered Multilinear Polynomials}\label{sec:centered}

The proof of the Theorem \ref{main} will follow from the application of the Markov's inequality to the upper bound on the $k$-th moment of the polynomial in question. The first step is to look at moments of ``centered'' multilinear  polynomials that replace $Y_v$ with $Y_v-\E\left[Y_v\right]$.
\begin{lemma}[Initial Moment Lemma]\label{lem:momentPrelim}
 We are given a hypergraph $H = ([n], {\cal H})$, $n$ independent central moment bounded random variables $Y=(Y_1,\dots, Y_n)$ with the same parameter $L>0$ and a polynomial $g(x)$
   with nonnegative coefficients $w_h \ge 0$ such that every monomial (or hyperedge) $h\in {\cal  H}$ has power (or cardinality) exactly $q$. We define random variables $X_v=Y_v-\E\left[Y_v\right]$ for $v\in [n]$.
   Then for any integer $k \ge 1$ we have
 \begin{eqnarray}\label{inequality44}
\left|\E\left[g(X)^k\right]\right| & \le & \max_{{\bar \sigma}} \left\{R_3^{qk} L^{qk-q\sigma_0-\ell}\cdot k^{qk-(q-1)\sigma_0-\ell}\cdot Var[g(X)]^{\sigma_0}\cdot \left(\prod_{t=1}^{q}\mu_t(g,Y)^{\sigma_t} \right) \right\}
\end{eqnarray}
where $R_3\ge 1$ is some absolute constant, $\ell = \sum_{t=0}^{q}(q-t)\sigma_t$, and ${\bar \sigma}=(\sigma_0,\dots,\sigma_q)$.
The maximum is over all non-negative integers $\sigma_t$, $0 \le t \le q$ satisfying  $2\sigma_0+\sum_{t=1}^q\sigma_t=k$ and $\ell \le qk/2$.
\end{lemma}

Note that the constraint $2\sigma_0+\sum_{t=1}^q\sigma_t=k$ in Lemma \ref{lem:momentPrelim} implies $\sigma_0 \le k/2$ hence the powers of $L$ and $k$  in (\ref{inequality44}) are non-negative.

\begin{proof}
Fix hypergraph $H=([n],{\cal H})$, random variables $Y=(Y_1,\dots,Y_n)$, non-negative weights $\{w_h\}_{h \in {\cal H}}$, an integer  $k$  and total power $q$.
 Without loss of generality we assume that ${\cal H}$ is the complete uniform hypergraph (setting additional edge weights to 0 as needed), i.e.\ ${\cal H}$ includes every possible hyperedge over vertex set $[n]$ with $q$ vertices. Note that the the cardinality of the hyperedge is equal to the total power of the corresponding monomial in the polynomial $g(x)$.

  A \emph{labeled} hypergraph $G=({\cal V}(G), {\cal H}(G))$ consists of a set of vertices ${\cal V}(G)$ and a \emph{sequence} of $k$ (not necessarily distinct) hyperedges ${\cal H}(G)=<h_1,\dots,h_k>$. In other words a labeled hypergraph is a hypergraph whose $k$ hyperedges are given unique labels from $[k]$.
We write e.g.\ $\prod_{h \in {\cal H}(G)} w_h$ as a shorthand for $\prod_{i=1}^k w_{h_i}$ where ${\cal H}(G)=<h_1,\dots,h_k>$; in particular duplicate hyperedges count multiple times in such a product.

Consider the sequence of hyperedges $h_1,\dots,h_k\in {\cal  H}$ from our original hypergraph $H$. These hyperedges define a labeled hypergraph $H(h_1,\dots,h_k)$ with vertex set $\cup_{i=1}^k {\cal V}(h_i)$ and hyperedge sequence $h_1,\dots,h_k$.
Note that the vertices of $H(h_1,\dots,h_k)$ are labeled by the indices from $[n]$ and the edges are labeled by the indices from $[k]$.
Note also that some hyperedges in $H(h_1,\dots,h_k)$ could span the same set of vertices, i.e.\ they are multiple copies of the same hyperedge in the original hypergraph $H$.
 Let ${\cal P}(H,k)$ be the set of all such edge and vertex labeled hypergraphs that can be generated by any $k$ hyperedges from $H$.
 We say that the \emph{degree} of a vertex (in a hypergraph) is the number of hyperedges it appears in.
Let ${\cal P}_2(H,k)\subseteq {\cal P}(H,k)$ be the set of such labeled hypergraphs where each vertex has degree at least two.
We split the whole proof into more digestible pieces by subsections.

\subsection{Changing the vertex labeling}

In this section we will show how to transform the formula for the $k$-th moment to have the summation over the hypergraphs that have its own set of labels instead of being labeled by the set $[n]$. Let $X_{v}=Y_v - \Eb{Y_v}$ for $v\in h$.
By linearity of expectation, independence of random variables $X_{v}$ for different vertices $v\in {\cal V}$ and definition of ${\cal P}(H,k)$ we obtain
\begin{eqnarray}
\left|\E\left[g(X)^k\right]\right|
&=&\left|\sum_{h_1,\dots,h_k\in {\cal  H}}\E\left[\prod_{i=1}^k\left(w_{h_i}\prod_{v\in \ve{h_i}}X_{v}\right)\right]\right|\nonumber \\
&=&\left|\sum_{G\in  {\cal P}(H,k)}\E\left[\prod_{h\in {\cal H}(G)}\left(w_{h}\prod_{v\in \ve{h}}X_{v}\right)\right]\right|\nonumber \\
&=&\left|\sum_{G\in  {\cal P}(H,k)}\left(\prod_{h\in {\cal H}(G)}w_{h} \right)\left(\prod_{v\in {\cal V}(G)}\E\left[\prod_{h\in {\cal H}(G)|v\in \ve{h}}X_{v}\right]\right)\right| \nonumber \\
&=&\left|\sum_{G\in  {\cal P}_2(H,k)}\left(\prod_{h\in {\cal H}(G)}w_{h} \right)\left(\prod_{v\in {\cal V}(G)}\E\left[\prod_{h\in {\cal H}(G)|v\in \ve{h}}X_{v}\right]\right) \right|  \label{inequality0.1} \\
&\le &\sum_{G\in  {\cal P}_2(H,k)}\left(\prod_{h\in {\cal H}(G)}w_{h} \right)\left(\prod_{v\in {\cal V}(G)}\left|\E\left[\prod_{h\in {\cal H}(G)|v\in \ve{h}}X_{v}\right]\right|\right) \nonumber \\
&\le &\sum_{G\in  {\cal P}_2(H,k)}\left(\prod_{h\in {\cal H}(G)}w_{h} \right)\left(\prod_{v\in {\cal V}(G)}\E\left[ \left|X_{v}\right|^{d_v(G)}\right]\right) , \label{inequality1}
\end{eqnarray}
where the   equality (\ref{inequality0.1}) follows from the fact that $\E[X_{v}]=0$  for all $v\in \ve{h}$ and $d_v(G)$ is the degree of vertex $v$ in the hypergraph $G$.

Note that a labeled hypergraph $G\in {\cal P}_2(H,k)$ could have the number of vertices ranging from $q$ up to $kq/2$ since every vertex has degree at least two. For $k$ and $q$  clear from context, let ${\cal S}_2(\ell)$ be the set of labeled hypergraphs
with vertex set $[\ell]$ having $k$ hyperedges such that each hyperedge has cardinality exactly  $q$  and every vertex has degree at least 2. For each hypergraph $G\in {\cal S}_2(\ell)$ the vertices are labeled by the indices from the set $[\ell]$ and the hyperedges are labeled by the indices from the set $[k]$.
  Let $M(S)$ for $S\subseteq [\ell]$ be the set of all possible injective functions $\pi:S\rightarrow [n]$, in particular
  $M([\ell])$ is the set of all possible injective functions $\pi:[\ell]\rightarrow [n]$. We will use the notation $\pi(h)$ for a copy of hyperedge $h \in {\cal H}(G)$ with its vertices relabeled by injective function $\pi$, i.e.\ ${\cal V}(\pi(h))=\{\pi(v) : v \in {\cal V}(h)\}$. We claim that
\begin{eqnarray}
\lefteqn{\sum_{G\in  {\cal P}_2(H,k)}\left(\prod_{h\in {\cal H}(G)}w_{h} \right)\left(\prod_{v\in {\cal V}(G)}\Eb{\left|X_{v}\right|^{d_v(G)}}\right)} \nonumber \\
&=& \sum_{\ell=q}^{kq/2}\frac{1}{\ell!}\sum_{G'\in {\cal S}_2(\ell)}\sum_{\pi\in M([\ell])} \left(\prod_{h\in {\cal H}(G')}w_{\pi(h)} \right)\left(\prod_{u\in {\cal V}(G')}\Eb{\left|X_{\pi(u)}\right|^{d_u(G')}}\right). \label{inequality2}
\end{eqnarray}
Indeed, every labeled hypergraph $G=({\cal V}(G),{\cal H}(G))\in  {\cal P}_2(H,k)$ on $\ell$ vertices has $\ell!$ labeled hypergraphs $G'=({\cal V}(G'),{\cal H}(G'))\in {\cal S}_2(\ell)$ that differ from $G$ by vertex labellings only. Each of those hypergraphs has one corresponding mapping $\pi$ that maps its $\ell$ vertex labels into vertex labels of hypergraph $G\in {\cal P}_2(H,k)$.

Then, combining (\ref{inequality1}) and (\ref{inequality2}) we obtain
\begin{eqnarray}\label{inequality3}
\left|\E\left[g(Y)^k\right]\right|\le \sum_{\ell=q}^{kq/2}\frac{1}{\ell!}\sum_{G'\in {\cal S}_2(\ell)} \sum_{\pi\in M([\ell])} \left(\prod_{h\in {\cal H}(G')}w_{\pi(h)} \right)\left(\prod_{u\in {\cal V}(G')}\Eb{\left|X_{\pi(u)}\right|^{d_u(G')} }\right).
\end{eqnarray}

\subsection{Estimating the term for each hypergraph $G'$}\label{sec:eachGraph}

We now fix integer $\ell$ and labeled hypergraph $G'\in {\cal S}_2(\ell)$.
Let $c$ be the number of \emph{connected components} in $G'$, i.e.\ $c$ is a maximal number such that the vertex set ${\cal V}(G')$ can be partitioned into $c$ parts ${\cal V}_1,\dots,{\cal V}_c$ such that for each hyperedge $h\in {\cal H}(G')$ and any $j\in [c]$ if $\ve{h}\cap {\cal V}_j\neq \emptyset$ then $\ve{h}\subseteq {\cal V}_j$. Intuitively, we can split the vertex set of $G'$ into $c$ components such that there are no hyperedges that have vertices in two or more components.
 By definition of degree $\sum_{v\in {\cal V}(G')}d_v=qk$  and $d_v\ge 2$ for all $v\in {\cal V}(G')$.

We use a canonical ordering $h^{(1)},\dots,h^{(k)}$ of the hyperedges in ${\cal H}(G')$ that will be specified later in Lemma \ref{lem:ordering}. (This canonical ordering is distinct from and should not be confused with the ordering of the hyperedges inherent in a labeled hypergraph.)
We iteratively remove hyperedges from the hypergraph $G'$ in this order. Let $G'_s=({\cal V}_s',{\cal H}_s')$ be the hypergraph defined by the hyperedges ${\cal H}_s'=h^{(s)},\dots,h^{(k)}$ and vertex set ${\cal V}_s'=\cup_{h \in {\cal H}_s'} \ve{h}$. In particular $G'_1$ is identical to $G'$ except for the order of the hyperedges.
 Let $V_s$ be the vertices of the hyperedge $h^{(s)}$ that have degree one in the hypergraph $G'_s$, i.e.\ ${\cal V}_{s+1}'={\cal V}_s'\setminus V_s$.  By definition, $0\le |V_s|\le q$. Moreover, $0\le |V_s|\le q-1$
 for $s=1,\dots,k-c$ by Lemma \ref{lem:ordering} since the hyperedge $h^{(s)}$ must be connected with at least one of remaining hyperedges.
By the properties of the canonical ordering $h^{(1)},\dots,h^{(k)}$ from Lemma \ref{lem:ordering} we know that the first $c$ edges (set $S_2$ of hyperedges) in that ordering belong to different connected components. Since degree of each node is at least two we obtain that $V_1=\dots=V_c=\emptyset$.

Analogously, we consider the second canonical ordering  ${\tilde h}^{(1)},\dots,{\tilde h}^{(k)}$ from Lemma \ref{lem:ordering} and define analogous notions.
Let ${\tilde G}'_s=({\tilde {\cal V}}_s',{\tilde{\cal H}}_s')$ be the hypergraph defined by the hyperedges ${\tilde {\cal H}}_s'={\tilde h}^{(s)},\dots,{\tilde h}^{(k)}$ and vertex set ${\tilde {\cal V}}_s'=\cup_{h \in {\tilde {\cal H}}_s'} \ve{h}$. Let ${\tilde V}_s$ be the vertices of the hyperedge ${\tilde h}^{(s)}$ that have degree one in the hypergraph ${\tilde G}'_s$, i.e.\ ${\tilde {\cal V}}_{s+1}'={\tilde {\cal V}}_s'\setminus {\tilde V}_s$.  The first $c$ edges in the second canonical  ordering ${\tilde h}^{(1)},\dots,{\tilde h}^{(k)}$ define the set $S_1$ of hyperedges that belong to different connected components. Therefore, ${\tilde V}_1=\dots={\tilde V}_c=\emptyset$.

 Let $S_1$ and $S_2$ be the sets of special hyperedges defined in Lemma \ref{lem:ordering}. Each set $S_i$ contains exactly one hyperedge per connected component of hypergraph $G'$ and therefore, hyperedges belonging to the same $S_i$ are disjoint. Let $W'_1=\cup_{h\in S_1} {\cal V}(h)$  be the set of vertices incident to    hyperedges in $S_1$ and $W'_2=\cup_{h\in S_2} {\cal V}(h)$  be the set of vertices incident to hyperedges in $S_2$.
Note that $|W_1'|=|W_2'|=qc$.

We apply the standard fact  $ab\le \frac{a^2+b^2}{2}$ to the $\prod_{h\in S_1\cup S_2}w_{\pi(h)}$ in the last term of the inequality (\ref{inequality3}) and obtain
\begin{eqnarray}\label{aux1}
 \left(\prod_{h\in S_1\cup S_2}w_{\pi(h)} \right)\le \frac{1}{2}\left(\prod_{h\in S_1}w^2_{\pi(h)}\right)+\frac{1}{2}\left(\prod_{h\in S_2}w^2_{\pi(h)} \right).
\end{eqnarray}

We will use the notation $d_u$ instead of $d_u(G')$. By central moment boundness, we estimate
\begin{align}\label{aux2}
\prod_{u\in {\cal V}(G')}\Eb{\left|X_{\pi(u)}\right|^{d_u } }
 &\le    \left(\prod_{u\in W_1'} \frac{d_u!}{2} L^{d_u-2}\Eb{ X_{\pi(u)}^2 }\right)
 \left(\prod_{u\in {\cal V}(G')\setminus W_1'} d_u!L^{d_u-1}\Eb{\left|X_{\pi(u)}\right|}\right) \nonumber \\
 &\le   2^{-qc}L^{qk-qc-\ell}  \left(\prod_{u\in {\cal V}(G')} d_u!  \right) \left(\prod_{u\in W_1'}  \Eb{ X_{\pi(u)}^2 } \right)
 \prod_{u\in {\cal V}(G')\setminus W_1'} \Eb{\left|X_{\pi(u)}\right|} \nonumber \\
 &\le 2^{\ell-2qc}L^{qk-qc-\ell}  \left(\prod_{u\in {\cal V}(G')} d_u!  \right) \left(\prod_{u\in W_1'}  \Eb{ X_{\pi(u)}^2 } \right)
 \prod_{u\in {\cal V}(G')\setminus W_1'} \Eb{\left|Y_{\pi(u)}\right|}
\end{align}
where the last inequality uses the inequality $\Eb{\left|X_{\pi(u)}\right|} = \Eb{\left|Y_{\pi(u)}-\Eb{Y_{\pi(u)}}\right|} \le \Eb{\left|Y_{\pi(u)}\right|} + \left|\Eb{Y_{\pi(u)}}\right| \le 2 \Eb{\left|Y_{\pi(u)}\right|}$.

Analogously,
\begin{eqnarray}\label{aux3}
\prod_{u\in {\cal V}(G')}\Eb{\left|X_{\pi(u)}\right|^{d_u } }\le  2^{\ell-2qc}L^{qk-qc-\ell}  \left(\prod_{u\in {\cal V}(G')} d_u!\right)\left( \prod_{u\in W_2'}  \Eb{ X_{\pi(u)}^2 }\right)
 \prod_{u\in {\cal V}(G')\setminus W_2'} \Eb{\left|X_{\pi(u)}\right|}.
\end{eqnarray}
Recall $[\ell]= {\cal V}(G')$, ${\cal V}_s'={\cal V}(G')\setminus \cup_{t=1}^{s-1} V_t$ for $s=1,\dots,k$ and $V_1=\dots=V_c=\emptyset$. Analogously,
${\tilde {\cal V}}_s'={\cal V}(G')\setminus \cup_{t=1}^{s-1} {\tilde V}_t$ for $s=1,\dots,k$ and ${\tilde V}_1=\dots={\tilde V}_c=\emptyset$.
For each $s=c+1,\dots,k-c$, we will use the notations
 \begin{eqnarray}\label{def}
 \Upsilon_s(\pi)&=& \left(\prod_{u\in W_1'}  \Eb{ X_{\pi(u)}^2 }\right)
 \prod_{u\in {\cal V}_s'\setminus W_1'} \Eb{\left|Y_{\pi(u)}\right|},\nonumber \\
 {\tilde \Upsilon}_s(\pi)&=&\left(\prod_{u\in W_2'}  \Eb{ X_{\pi(u)}^2 }\right)
 \prod_{u\in {\tilde {\cal V}}_s(G')\setminus W_2'} \Eb{\left|Y_{\pi(u)}\right|}.
\end{eqnarray}

Therefore, combining inequalities  (\ref{aux1}), (\ref{aux2}), (\ref{aux3}) and notations (\ref{def}), for each graph $G'\in {\cal S}_2(\ell)$,  we obtain
 \begin{eqnarray}\label{transformed}
&&\sum_{\pi\in M([\ell])} \left(\prod_{h\in {\cal H}(G')}w_{\pi(h)} \right)\left(\prod_{u\in {\cal V}(G')}\Eb{\left|X_{\pi(u)}\right|^{d_u(G')} }\right)\nonumber\\
&&\le 2^{\ell-2qc-1}L^{qk-qc-\ell}\left( \prod_{u\in {\cal V}(G')} d_u!\right) \cdot \nonumber \\
&&\left(\sum_{\pi\in M([\ell])} \left(\prod_{h\in S_1}w^2_{\pi(h)}\right)\cdot \left( \prod_{h=h^{(c+1)},\dots,h^{(k-c)}} w_{\pi(h)} \right)\Upsilon_1(\pi) \right. \nonumber \\
&&+\left. \sum_{\pi\in M([\ell])} \left(\prod_{h\in S_2}w^2_{\pi(h)} \right)\cdot \left( \prod_{h=\tilde h^{(c+1)},\dots,\tilde h^{(k-c)}}w_{\pi(h)} \right){\tilde \Upsilon}_1(\pi) \right) \nonumber \\
\end{eqnarray}

We now analyze two terms in the inequality (\ref{transformed}) separately using different canonical orderings of the hyperedges from Lemma \ref{lem:ordering}.
We consider $\sum_{\pi\in M([\ell])} \left(\prod_{h\in S_1}w^2_{\pi(h)}\right)\cdot \left( \prod_{h=h^{(c+1)},\dots,h^{(k-c)}}w_{\pi(h)} \right)\Upsilon_1(\pi)$ and the corresponding canonical ordering of the hyperedges
$h^{(c+1)},\dots,h^{(k)}$. For each $s=c+1,\dots,k-c$ we obtain
\begin{eqnarray}
&&\sum_{\pi\in M({\cal V}_s')} \left(\prod_{h\in S_1} w^2_{\pi(h)} \right)\left(\prod_{h=h^{(s)},\dots,h^{(k-c)}}w_{\pi(h)} \right) \Upsilon_s(\pi)\nonumber \\
&&= \sum_{\pi'\in M({\cal V}_{s+1}')} \sum_{\substack{\pi\in M({\cal V}_s') \\ \text{s.t. } \pi \text{ extends } \pi'}}
 \left(\prod_{h\in S_1} w^2_{\pi(h)} \right) \left(\prod_{h=h^{(s+1)},\dots,h^{(k-c)}}w_{\pi(h)} \right)  \Upsilon_{s+1}(\pi) \left(w_{\pi(h^{(s)})}\prod_{v\in V_s} \Eb{\left|Y_{\pi(v)}\right|}\right) \nonumber \\
&&= \sum_{\pi'\in M({\cal V}_{s+1}')}\left(\prod_{h\in S_1} w^2_{\pi'(h)} \right)\left(\prod_{h=h^{(s+1)},\dots,h^{(k-c)}}w_{\pi'(h)} \right)\Upsilon_{s+1}(\pi') \left[  \sum_{\substack{\pi\in M({\cal V}_s') \\ \text{s.t. } \pi \text{ extends } \pi'}}
\left(w_{\pi(h^{(s)})}\prod_{v\in V_s}\Eb{\left|Y_{\pi(v)}\right|} \right)\right]
\nonumber
\end{eqnarray}
where we say that $\pi$ \emph{extends} $\pi'$ if $\pi(v)=\pi'(v)$ for every $v$ in the domain of $\pi'$.

 We now group the sum over $\pi$ by the value of $\pi(h^{(s)}) \equiv h \in {\cal H}$. Note that for any fixed mapping $\pi'\in M({\cal V}_{s+1}')$ there are exactly  $|V_s|!$ possible mappings $\pi \in M({\cal V}_s')$ that extend $\pi'$ and map the vertex labels of hyperedge $h^{(s)}\in G'$ into vertex labels of the hyperedge $h\in {\cal H}$.
  Let $S'=\{ \pi'(v):     v\in  {\cal V}(h^{(s)}) \setminus V_s\}$, which is the portion of $\pi({\cal V}(h^{(s)}))$ that is fixed by $\pi'$. Then
\begin{eqnarray*}
\sum_{\substack{\pi\in M({\cal V}_s') \\ \text{s.t. } \pi \text{ extends } \pi'}} w_{\pi(h^{(s)})}\prod_{v\in V_s} \Eb{\left|Y_{\pi(v)}\right|} &=& |V_s|! \sum_{h \in {\cal H}:{\cal V}(h) \supseteq S'} w_{h} \prod_{u\in \ve{h} \setminus S'} \Eb{\left|Y_{u}\right|}\\
&\le& |V_s|! \max_{S:|S|=q-|V_s|}\left\{\sum_{h \in {\cal H}:{\cal V}(h) \supseteq S} w_h\prod_{u\in \ve{h}\setminus S}\Eb{\left|Y_{u}\right|}\right\} \\
&=& |V_s|! \mu_{q-|V_s|}(g,Y).
\end{eqnarray*}

For $s=k-c+1,\dots,k$, ${\cal H}_s'\subseteq S_1$, $|V_s|=q$ and we have a similar argument (note that $|M({\cal V}_{k+1}')|=|M(\emptyset)|=1$)
\begin{eqnarray}
&&\sum_{\pi\in M({\cal V}_s')} \left(\prod_{h\in {\cal H}_s'}w^2_{\pi(h)} \right)\left(\prod_{u\in  {\cal V}_s'}\Eb{X_{\pi(u)}^2}\right) \nonumber \\
&&= \sum_{\pi'\in M({\cal V}_{s+1}')} \sum_{\substack{\pi\in M({\cal V}_s') \\ \text{s.t. } \pi \text{ extends } \pi'}}
\left(\prod_{h\in {\cal H}_{s+1}'}w^2_{\pi(h)} \right)\left(\prod_{u\in  {\cal V}_{s+1}'}\Eb{X_{\pi(u)}^2}\right)
\left(w^2_{\pi(h^{(s)})}\prod_{u\in V_s} \Eb{X_{\pi(u)}^2}\right) \nonumber \\
&&= \sum_{\pi'\in M({\cal V}_{s+1}')}\left[  \left(\prod_{h\in {\cal H}_{s+1}'}w^2_{\pi'(h)} \right)\left(\prod_{u\in  {\cal V}_{s+1}'}\Eb{X_{\pi'(u)}^2}\right)
 \sum_{\substack{\pi\in M({\cal V}_s') \\ \text{s.t. } \pi \text{ extends } \pi'}}
\left(w^2_{\pi(h^{(s)})}\prod_{u\in V_s} \Eb{X_{\pi(u)}^2}\right)\right]\nonumber \\
&&= |V_s|! \cdot \sum_{\pi'\in M({\cal V}_{s+1}')}\left[  \left(\prod_{h\in {\cal H}_{s+1}'}w^2_{\pi'(h)} \right)\left(\prod_{u\in  {\cal V}_{s+1}'}\Eb{X_{\pi'(u)}^2}\right)\right]
\cdot \left(\sum_{h \in {\cal H}} w^2_{h} \prod_{u\in \ve{h}} \Eb{ X_{u}^2} \right) \nonumber \\
&&= |V_s|!\cdot \sum_{\pi'\in M({\cal V}_{s+1}')}\left[  \left(\prod_{h\in {\cal H}_{s+1}'}w^2_{\pi'(h)} \right)\left(\prod_{u\in  {\cal V}_{s+1}'}\Eb{X_{\pi'(u)}^2}\right)\right]\cdot Var[g(X)], \nonumber
\end{eqnarray}
where in the last equality we used Lemma \ref{variance}.

In the end we bound the first term of (\ref{transformed}) as follows:
\begin{eqnarray} \label{estimate}
&&2^{\ell-2qc-1}L^{qk-qc-\ell} \left( \prod_{u\in {\cal V}(G')} d_u!\right)\sum_{\pi\in M([\ell])} \left(\prod_{h\in S_1}w^2_{\pi(h)}\right)\cdot \left( \prod_{h=h^{(c+1)},\dots,h^{(k-c)}}w_{\pi(h)} \right)\Upsilon_1(\pi) \nonumber \\
&\le& 2^{\ell-2qc-1}L^{qk-qc-\ell}\left(\prod_{v\in {\cal V}(G')} d_v!\right)\left(\prod_{s=k-c+1}^{k} |V_s|!Var[g(X)]\right)\prod_{s=c+1}^{k-c} \left(|V_s|!  \mu_{q-|V_s|} \right)\nonumber\\
&=& 2^{\ell-2q\sigma_0-1}L^{qk-q\sigma_0-\ell}\left(\prod_{v\in {\cal V}(G')} d_v!\right)\left(\prod_{s=1}^k |V_s|!  \right)Var[g(X)]^{\sigma_0}\prod_{t=1}^{q}\mu_t^{\sigma_t}\nonumber\\
&\le& 2^{\ell-2q\sigma_0-1}L^{qk-q\sigma_0-\ell}\left(\prod_{v\in {\cal V}(G')} d_v!\right) Var[g(X)]^{\sigma_0}q^\ell \cdot \prod_{t=1}^{q}\mu_t^{\sigma_t}
\end{eqnarray}
where $\sigma_0=c$, $\sigma_t$ for $t\ge 1$ is the number of indices $s=c+1,\dots,k-c$ with $q-|V_s|=t$ and  $\mu_t = \mu_t(w,Y)$. In the last inequality we used the fact that $\sum_{s=1}^k|V_s|=\ell$ and $|V_s|\le q$.
The quantities $\sigma_t$ must satisfy the equalities $\sum_{t=0}^{q} (q-t)\sigma_t = \ell$ and $2\sigma_0+\sum_{t=1}^q\sigma_t=k$.

Using analogous argument for the canonical ordering ${\tilde h}^{(1)},\dots,{\tilde h}^{(k)}$ we show
\begin{eqnarray} \label{estimate1}
&&2^{\ell-2qc-1}L^{qk-qc-\ell}  \left(\prod_{u\in {\cal V}(G')} d_u!\right)\sum_{\pi\in M([\ell])} \left(\prod_{h\in S_2}w^2_{\pi(h)} \right)\cdot \left( \prod_{h=\tilde h^{(c+1)},\dots,\tilde h^{(k-c)}}w_{\pi(h)} \right){\tilde \Upsilon}_1(\pi)  \nonumber\\
&&\le 2^{\ell-2q\sigma'_0-1}L^{qk-q\sigma'_0-\ell}\left(\prod_{v\in {\cal V}(G')}d_v!\right) Var[g(X)]^{\sigma'_0}q^\ell \cdot \prod_{t=1}^{q}\mu_t^{\sigma'_t},
\end{eqnarray}
where $\sigma_0'=c$ and $\sigma_1',\dots,\sigma_q'$ is a different collection of powers satisfying conditions of the Lemma.

Combining the inequalities (\ref{transformed}), (\ref{estimate})  and  (\ref{estimate1}) we derive
\begin{eqnarray}\label{lastestimate}
&&\sum_{\pi\in M([\ell])} \left(\prod_{h\in {\cal H}(G')}w_{\pi(h)} \right)\left(\prod_{u\in {\cal V}(G')}\Eb{\left|X_{\pi(u)}\right|^{d_u(G')} }\right)\nonumber \\
&&\le\max_{{\bar \sigma}} \left\{ 2^{\ell-2q\sigma_0}L^{qk-q\sigma_0-\ell}\left(\prod_{v\in {\cal V}(G')} d_v!\right) Var[g(X)]^{\sigma_0}q^\ell \cdot \prod_{t=1}^{q}\mu_t^{\sigma_t}\right\}
\end{eqnarray}
where the maximum is over all non-negative integers $\sigma_0,\sigma_1,\dots,\sigma_q$ satisfying $2\sigma_0+\sum_{t=1}^q\sigma_t=k$, $ \sigma_0=c$ and $\ell = \sum_{t=0}^{q}(q-t)\sigma_t$.
\subsection{Using the Counting Lemma}\label{sec:useCount}
We decompose ${\cal S}_2(\ell)$ as ${\cal S}_2(\ell) = \bigcup_{c,\bar{d} \ge \bar{2}}{\cal S}(\ell, c, \bar{d})$ where $\bar{2}$ is a vector of $\ell$ twos and ${\cal S}(\ell,c,\bar{d})$ is the number of vertex and hyperedge labeled  hypergraphs with vertex set $[\ell]$ and $k$ labeled hyperedges such that each hyperedge has cardinality $q$, the number of connected components is $c$, the degree vector is $\bar{d}$. (Note that ${\cal S}(\ell, c, \bar{d})$ depends on $k$ and $q$ as well.)
Let ${\bar \sigma}=(\sigma_0,\dots,\sigma_q)$. Combining, (\ref{inequality3}) and (\ref{lastestimate}) we obtain
\begin{eqnarray}
&&\left|\E\left[g(X)^k\right]\right|
 \le  \sum_{\ell=q}^{kq/2}\frac{1}{\ell!}\sum_{c=1}^{\ell/q}\sum_{\bar{d}\ge \bar{2}}\left( \sum_{G'\in {\cal S}(\ell,c,\bar{d})} \max_{{\bar \sigma}} \left\{2^{\ell-2q\sigma_0} L^{qk-q\sigma_0-\ell} Var[g(X)]^{\sigma_0} \left(\prod_{t=1}^{q}\mu_t^{\sigma_t} \right) q^{\ell} \prod_{v \in [\ell]} d_v!\right\}\right) \nonumber \\
&& \le  \max_{\ell,c,\bar{d}\ge \bar{2},{\bar \sigma}}\left\{ \frac{kq}{2}\cdot \frac{1}{\ell!} \cdot\frac{\ell}{q}\cdot 2^{q k + \ell} \cdot  |{\cal S}(\ell,c,\bar{d})| \cdot 2^{\ell-2q\sigma_0} L^{qk-q\sigma_0-\ell} Var[g(X)]^{\sigma_0} \left(\prod_{t=1}^{q}\mu_t^{\sigma_t} \right)q^{\ell} \prod_v d_v! \right\} \nonumber \\
&& \le   \max_{\ell,c,\bar{d}\ge \bar{2},{\bar \sigma}} \Bigg\{\frac{k\ell}{2\cdot \ell !}\cdot 2^{q(k-2\sigma_0) + 2\ell} \cdot L^{qk-q\sigma_0-\ell} Var[g(X)]^{\sigma_0}  \cdot q^{\ell} \cdot \underbrace{R_0^{qk}\cdot k^{qk-(q-1)c}}_{\substack{|{\cal S}(\ell,c,\bar{d})|\prod_v d_v! ~\le~ \text{this}\\ \text{by counting Lemma \ref{MainCount}}}}\cdot \left(\prod_{t=1}^{q}\mu_t^{\sigma_t} \right) \Bigg\}\nonumber
\end{eqnarray}
where the sum is over $\bar d \ge 2$ with $\sum_{v \in [\ell]} d_v = qk$ and the maximum over $\bar \sigma$ has the same constraints as in (\ref{lastestimate}). The maximums over $\ell$, $c$, and $\bar d$ are over the same sets that those quantities were previously summed over.
The second inequality follows from the fact that the total number of feasible degree vectors $\bar{d}$ is at most $2^{q k+\ell}$ ($q k$ is the sum of all the degrees and we need to compute the total number of partitions of the array with $q k$ entries into $\ell$ possible groups of consecutive entries which is ${q k+\ell-1 \choose \ell-1}$).

We now substitute $\sigma_0$ for $c$, and remove the unreferenced variables $c$ and $\bar{d}$  from the maximum.  We also remove $\ell$ from the maximum since it is completely defined by the vector ${\bar \sigma}$. We continue
\begin{eqnarray}
\left|\E\left[g(X)^k\right]\right|
& \le &  \max_{{\bar \sigma}} \left\{\frac{k\ell}{\ell !}\cdot L^{qk-q\sigma_0-\ell} Var[g(X)]^{\sigma_0}  \cdot q^{\ell} \cdot R_1^{qk} \cdot k^{qk-(q-1)\sigma_0 }\cdot \left(\prod_{t=1}^{q}\mu_t^{\sigma_t} \right) \right\}\nonumber \\
& \le &  \max_{{\bar \sigma} } \left\{R_2^{qk} \cdot \left(\frac{q}{\ell}\right)^{\ell}\cdot L^{qk-q\sigma_0-\ell} Var[g(X)]^{\sigma_0} \cdot k^{qk-(q-1)\sigma_0 }\cdot \left(\prod_{t=1}^{q}\mu_t^{\sigma_t} \right) \right\}\nonumber \\
& = &  \max_{{\bar \sigma} } \left\{R_2^{qk} \cdot \left(\frac{q k}{\ell}\right)^{\ell} \cdot  L^{qk-q\sigma_0-\ell} Var[g(X)]^{\sigma_0} \cdot k^{qk-(q-1)\sigma_0-\ell}\cdot \left(\prod_{t=1}^{q}\mu_t^{\sigma_t} \right) \right\}\nonumber \\
& \le &  \max_{{\bar \sigma}} \left\{R_3^{qk}  L^{qk-q\sigma_0-\ell} Var[g(X)]^{\sigma_0} \cdot k^{qk-(q-1)\sigma_0-\ell}\cdot \left(\prod_{t=1}^{q}\mu_t^{\sigma_t} \right) \right\}\label{UsingCount}
\end{eqnarray}
where $R_0<R_1<R_2<R_3$ are some absolute constants, the second inequality uses the fact  that $\ell!\ge (\ell/e)^\ell$, and the last inequality is implied by the fact that
\[
\left(\frac{qk}{\ell}\right)^{\ell} \le \max_{x>0} \left(\frac{kq}{x}\right)^{x} = e^{qk/e}.
\]
Inequality (\ref{UsingCount}) is precisely the inequality (\ref{inequality44}) that we needed to prove.
\end{proof}

\section{Intermediate moment lemma}\label{sec:intermediate}

\begin{lemma}[Intermediate Moment Lemma]\label{lem:moment}
We are given $n$ independent central moment bounded random variables $Y=(Y_1,\dots, Y_n)$ with the same parameter $L>0$  and a general polynomial $g(x)$ with nonnegative coefficients such that every monomial (or hyperedge) $h\in {\cal  H}$ has power exactly $q$. Let $X_v=Y_v-\Eb{Y_v}$ then
 \begin{eqnarray}\label{inequality4}
\left|\E\left[g(X)^k\right]\right| & \le & \max \left\{ \left(\sqrt{kR_3^q  Var[g(X)]}\right)^{k}, \max_{t\in [q]}(k^{t} R_3^q L^t \mu_t(g,Y))^{k}\right\}.
\end{eqnarray}
where $R_3\ge 1$ is some absolute constant and $X$ is the vector of centered random variables defined in the Lemma \ref{lem:momentPrelim}.
\end{lemma}

\begin{proof}
First we note that $2\sigma_0+\sum_{t=1}^q\sigma_t=k$ and $\sum_{t=0}^{q}(q-t)\sigma_t =\ell$ imply,
\begin{eqnarray*}
\sum_{t=1}^{q}t\sigma_t&=&qk-q(2\sigma_0+\sum_{t=1}^q\sigma_t)+\sum_{t=1}^{q}t\sigma_t\\
&=&qk-q\sigma_0-\sum_{t=0}^{q}(q-t)\sigma_t\\
&=&qk-q\sigma_0-\ell.
\end{eqnarray*}
Therefore, $\sigma_0+\sum_{t=1}^{q}t\sigma_t=qk-(q-1)\sigma_0-\ell$.
Combining these facts with Lemma \ref{lem:momentPrelim}, we derive
\begin{eqnarray*}
\left|\E\left[g(X)^k\right]\right|&\le &\max_{{\bar \sigma}} \left\{R_3^{qk} L^{qk-q\sigma_0-\ell}Var[g(X)]^{\sigma_0} \cdot k^{qk-(q-1)\sigma_0-\ell}\cdot \left(\prod_{t=1}^{q}\mu_t^{\sigma_t} \right) \right\}\nonumber \\
&\le&     \max_{{\bar \sigma}} \left\{(k R_3^q Var[g(X)])^{\sigma_0}\cdot \prod_{t=1}^{q}(k^{t}R_3^q L^t\mu_t)^{\sigma_t}\right\} \nonumber \\
 & = &  \max_{{\bar \sigma}} \left\{ \left(\sqrt{k R_3^q Var[g(X)]}\right)^{2\sigma_0}\cdot \prod_{t=1}^{q}(k^{t}R_3^q L^t\mu_t)^{\sigma_t}\right\}\\
&\le&  \max \left\{ \left(\sqrt{k R_3^q Var[g(X)]}\right)^{k}, \max_{t\in [q]}\left(k^{t} R_3^q L^t \mu_t\right)^{k}\right\},
\end{eqnarray*}
where the last inequality is based on the fact that $2\sigma_0+\sum_{t=1}^q\sigma_t=k$.
\end{proof}

\section{Technical Lemmas}\label{technical}
\begin{lemma}\label{variance}
We are given $n$ independent random variables $X_1,\dots, X_n$ such that $\Eb{X_i}=0$ for all $i\in [n]$. We are also given a multilinear polynomial $f(x)= \sum_{h\in {\cal  H}(H) } w_h\prod_{v\in {\cal V}(h)}x_v$ with $|{\cal V}(h)|\ge 1$ for any $h\in {\cal  H}(H)$. Then
$$Var[f(X)]= \sum_{h\in {\cal  H}(H) } w_h^2\prod_{v\in {\cal V}(h)}\Eb{X^2_v}.$$
\end{lemma}
\begin{proof}
Let ${\cal H} = {\cal H}(H)$.
Clearly $\Eb{f(X)}=0$ hence
\begin{align*}
Var[f(X)] & = \Eb{(f(X) - \Eb{f(X)})^2} = \Eb{f(X)^2} = \Eb{\left(\sum_{h\in {\cal H}} w_h \prod_{v \in h} X_v\right)^2} \\
& = \sum_{h\in {\cal H}} \sum_{h'\in {\cal H}} w_h w_{h'} \Eb{\left(\prod_{v \in h} X_v\right)\left(\prod_{v \in h'} X_v\right) } \\
& = \sum_{h\in {\cal H}} \sum_{h'\in {\cal H}} w_h w_{h'} \left(\prod_{v \in h \cap h'} \Eb{X_v^2}\right) \prod_{v \in (h \setminus h') \cup (h' \setminus h)} \Eb{X_v}\\
& = \sum_{h\in {\cal  H}} w_h^2\prod_{v\in h}\Eb{X^2_v}
\end{align*}
where the last equality follows because $\Eb{X_v}=0$ by assumption.
\end{proof}

\begin{lemma}\label{orthogonal}
We are given $n$ independent random variables $X_1,\dots, X_n$ such that $\Eb{X_i}=0$ for all $i\in [n]$. We are also given  two multilinear polynomials $g_1(x)= \sum_{h\in {\cal  H}(H) } w_h\prod_{v\in h}x_v$
and $g_2(x)= \sum_{h\in {\cal  H}(H) } w'_h\prod_{v\in h}x_v$ such that  $w_hw'_h=0$ for any hyperedge $h$. Then
$\Eb{g_1(X)g_2(X)}=0.$
\end{lemma}
\begin{proof}
The proof is similar to the proof of Lemma \ref{variance}:
\begin{align*}
\Eb{g_1(X)g_2(X)} & = \sum_{h\in {\cal  H}(H)} \sum_{h' \in {\cal  H}(H)} w_h w'_{h'} \Eb{\left(\prod_{v\in h} X_v\right)\left(\prod_{v\in h'} X_v\right)} \\
& = \sum_{h\in {\cal  H}(H)} \sum_{h' \in {\cal  H}(H)} w_h w'_{h'}\left(\prod_{v\in h \cap h'}\Eb{X_v^2}\right)\left(\prod_{v\in (h \setminus h') \cup (h' \setminus h)}\Eb{X_v}\right) \\
& = 0
\end{align*}
where the final equality follows because each $h,h'$ term either has $w_h w'_{h'} = 0$ or else $h \ne h'$ and hence at least one $\Eb{X_v}=0$ factor.
\end{proof}

\begin{lemma}[Ordering Lemma]\label{lem:ordering}
We are given a hypergraph $G'=({\cal V},{\cal H})$ with $c$ connected components and degree of each vertex $\ge 2$. We can define two disjoint sets of hyperedges  $S_1$ and $S_2$ each containing exactly one hyperedge per connected component of $G'$, i.e. $|S_1|=|S_2|=c$, such that there exist two canonical orderings  $h^{(1)},\ldots,h^{(k)}$ and ${\tilde h}^{(1)},\ldots,{\tilde h}^{(k)}$ of the hyperedges ${\cal H}$ with the following properties:
\begin{enumerate}
\item $S_2=\{h^{(1)},\ldots,h^{(c)}\}=\{{\tilde h}^{(k-c+1)},\ldots,{\tilde h}^{(k)}\}$ and  $S_1=\{h^{(k-c+1)},\ldots,h^{(k)}\}=\{{\tilde h}^{(1)},\ldots,{\tilde h}^{(c)}\}$, i.e. the hyperedges from $S_2$ appear first in the canonical ordering $h^{(1)},\ldots,h^{(k)}$ and last in the canonical ordering ${\tilde h}^{(1)},\ldots,{\tilde h}^{(k)}$, while  the hyperedges from $S_1$ appear last in the canonical ordering $h^{(1)},\ldots,h^{(k)}$ and first in the canonical ordering ${\tilde h}^{(1)},\ldots,{\tilde h}^{(k)}$;
\item for any  $s= 1,\ldots,k-c$ the hypergraph $G_s$ induced by the hyperedges $h^{(s)},\ldots,h^{(k)}$ has exactly  $c$ connected components;
\item  Analogously, for any $s=1,\ldots,k-c$ the hypergraph ${\tilde G}_s$ induced by the hyperedges ${\tilde h}^{(s)},\ldots,{\tilde h}^{(k)}$ has exactly $c$ connected components.
\end{enumerate}
\end{lemma}
\begin{proof}
Let ${\cal L}$ be the line graph of $G'$, i.e.\ an undirected graph with one vertex for each of the $k$ hyperedges of $G'$ and an edge connecting every pair of vertices that correspond to hyperedges with intersecting vertex sets. Pick an arbitrary spanning forest ${\cal F}$ of ${\cal L}$. Pick two leaves arbitrarily from each connected component of ${\cal F}$ and arbitrarily put one from each component in $S_1$ and the others in $S_2$. The existence of at least two leaves in each component follows because all vertices of $G'$ have degrees at least 2 and hence each connected component has at least two hyperedges. It is easy to see that any tree with at least two vertices has at least two leaves.\footnote{Indeed root each tree arbitrarily; if the root has degree two or more pick arbitrary leaf descendents (in the rooted sense) of two neighbors of the root; otherwise pick the root and an arbitrary leaf descendent of the root.}

We show the construction of $h^{(1)},\ldots,h^{(k)}$ only; the construction of ${\tilde h}^{(1)},\ldots,{\tilde h}^{(k)}$ is analogous with the roles of $S_1$ and $S_2$ swapped.

We pick $h^{(1)},\ldots,h^{(k)}$ iteratively (in that order) as follows. Let ${\cal F}_i$ denote the subforest of ${\cal F}$ induced by vertices ${\cal H} \setminus \{h^{(1)},\dots,h^{(i-1)}\}$. We pick $h^{(i)}$ to be an arbitrary leaf of ${\cal F}_i$ subject to the constraint that $h^{(i)} \in S_2$ if $1 \le i \le c$, $h^{(i)} \not \in S_1 \cup S_2$ if $c+1 \le i \le k-c$, and $h^{(i)} \in S_1$ if $k-c+1 \le i \le k$.

For any $1 \le i \le k$ we assert that:
\begin{enumerate}
\item there is a leaf satisfying the desired constraint available to be $h^{(i)}$ and
\item if $i \le k-c$ there are $c$ connected components of ${\cal F}_{i+1}$ and each contains a vertex (hyperedge of $G'$) in $S_1$.
\end{enumerate}

The second property follows because we always choose a leaf and never choose a vertex from $S_1$.

For $1 \le i \le c$ the first property follows because removing a vertex from a graph cannot make a leaf into a non-leaf and every vertex of $S_2$ is a leaf in ${\cal F}_1$. For $c+1 \le i \le k-c$ the first property follows because the second property implies that there is a connected component of ${\cal F}_i$ with at least two vertices and hence leaves and at most one can be from $S_1$ and none from $S_2$.
\end{proof}
The next lemma was proven in \cite{SS} in the setting of general polynomials. We state below a special case corresponding to multilinear polynomials.
\begin{lemma}[Main Counting Lemma \cite{SS}]\label{MainCount}
For any $k$, $q\ge 1$, $\ell$, $c$ and $\bar{d} \ge \bar{2}$  we have
\begin{align*}
|{\cal S}(\ell, c, \bar{d})| \left(\prod_{v\in [\ell]} d_v!\right) &\le   R_0^{qk}   k^{qk - (q-1)c},
\end{align*}
for some universal constant $R_0>1$.
\end{lemma}

\begin{lemma}[H\"older's Inequality]\label{holder}
Let $p_1,\dots,p_k\in (1,+\infty)$ such that $\sum_{i=1}^k\frac{1}{p_i}=1$ then for arbitrary collection $X_1,\dots,X_k$ of random variables on the same probability space the following inequality holds
$$\E\left[\left|\prod_{i=1}^kX_i\right|\right]\le \prod_{i=1}^k\E\left[\left|X_i\right|^{p_i}\right]^{1/p_i}.$$
\end{lemma}
We will use the following corollary of H\"older's inequality.
\begin{corollary}[Minkowski's Inequality]\label{cor:momentOfSum}
Let $k$ be a positive integer and $Z_1,Z_2,\ldots,Z_m$ be (potentially dependent) random variables with $\E[|Z_i|^k] \le z_i^k$ for $z_i\in R_+$. It follows that
\begin{align}
\E\left[ \left(\sum_{i=1}^m|Z_i|\right)^k\right] & \le \left(\sum_{i=1}^mz_i\right)^k \label{eqn:momentOfSum}
.\end{align}
\end{corollary}

\section{General Even Moment Lemma}\label{ProofMainGen}

\begin{lemma}[General Even Moment Lemma]\label{generalmoment}
We are given $n$ independent central moment bounded random variables $Y=(Y_1,\dots, Y_n)$ with the same  parameter $L>0$ and a general power $q$ polynomial $f(x)$. Let $k\ge 2$ be an even integer then
 \begin{eqnarray}\label{inequality4gen}
\E\left[\left| f(Y)-\E\left[ f(Y)\right]\right|^k\right] & \le & \max \left\{ \left(\sqrt{k R_4^q Var[f(Y)]}\right)^{k}, \max_{t\in [q]}\left(k^{t}R_4^q  L^t\mu_t(f,Y)\right)^{k}\right\}.
\end{eqnarray}
where $R_4\ge 1$ is some absolute constant.
\end{lemma}

\begin{proof}
Let weight function $w$ and hypergraph $H=([n],{\cal H})$ be such that $f(Y) = \sum_{h \in {\cal H}} w_h \prod_{v \in \ve{h}} Y_v$. Let $X_{v} = Y_v-\Eb{Y_v}$.
Let ${\cal H}'$ denote the set of all possible hyperedges (including the empty hyperedge) with at most $q$ vertices (from ${\cal V}(H)=[n]$).
First we note that
\begin{align}
f(Y) &= \sum_{h \in {\cal H}} w_h \prod_{v \in \ve{h}} (X_{v} + \Eb{Y_v}) \nonumber \\
& = \sum_{h' \in {\cal H}'} \sum_{h \in {\cal H} : \ve{h} \supseteq \ve{h'}} w_h \left(\prod_{v \in \ve{h} \setminus \ve{h'}} \Eb{Y_v}\right)\left( \prod_{v \in \ve{h'}} X_{v}\right) \nonumber \\
& = \sum_{h' \in {\cal H}'} w'_{h'} \prod_{v \in \ve{h'}} X_{v} \label{eqn:center}
\end{align}
where
\[
w'_{h'} = \sum_{h \in {\cal H} : \ve{h} \supseteq \ve{h'}}w_h\left(\prod_{v \in \ve{h} \setminus \ve{h'}} \Eb{Y_v}\right)
.\]
We next group the monomials on the right hand side of (\ref{eqn:center}) by cardinality  and sign of coefficient, yielding $m \le 2q$ polynomials $g^{(1)},\dots,g^{(m)}$ with corresponding weight functions for all monomials $w^{(1)},\dots,w^{(m)}$ and powers $q_1,\dots,q_m$. That is,
\begin{align}
f(Y) & = w'_{\{\}} + \sum_{i=1}^m \sum_{h':|h'| \ge 1} w^{(i)}_{h'} \prod_{v \in \ve{h'}} X_{v} \\
& = \Eb{f(Y)} + \sum_{i=1}^m g^{(i)}(X)
\nonumber
\end{align}
where $\{\}$ is the empty hyperedge.
We have
\begin{align}
\mu_r(w^{(i)}, Y) &\le \mu_r(w', Y) = \max_{S:|S|=r} \sum_{h':{\cal V}(h') \supseteq S} |w'_{h'}| \prod_{v \in \ve{h'} \setminus S} \Eb{|Y_v|} \nonumber\\
 & \le \max_{S:|S|=r} \sum_{h': {\cal V}(h') \supseteq S} \sum_{h \in {\cal H} : \ve{h} \supseteq \ve{h'}}|w_h|\left(\prod_{v \in \ve{h} \setminus \ve{h'}} \Eb{|Y_v|}\right) \prod_{v \in \ve{h'} \setminus \ve{h_0}} \Eb{|Y_v|} \nonumber\\
 & \le  2^{q}\max_{S:|S|=r} \sum_{h: {\cal V}(h) \supseteq S} |w_h|\left(\prod_{v \in \ve{h} \setminus \ve{h_0}} \Eb{|Y_v|}\right) = 2^{q}\mu_r(w,Y)=2^{q}\mu_r, \label{eqn:muiMu}
\end{align}
where the last inequality follows from the fact that for any $h\in {\cal H}$ the number of different $h'\in {\cal H}'$ such that $\ve{h} \supseteq {\cal V}(h') \supseteq S$ is at most $2^{q}$.
In addition by Lemma \ref{orthogonal} we derive
\begin{eqnarray}
Var[f(Y)]&=&\sum_{i=1}^m Var[g^{(i)}(X)]. \label{eqn:varSumVars}
\end{eqnarray}

For even $k\ge 2$ Lemma \ref{lem:moment} implies that
$$\E\left[\left|g^{(i)}(X)\right|^k\right]=\left|\E\left[g^{(i)}(X)^k\right]\right|\le \max \left\{ \left(\sqrt{kR_3^{q_i}  Var[g^{(i)}(X)]}\right)^{k},
\max_{t\in [q_i]}(k^{t}R_3^{q_i} L^t\mu_t(w^{(i)},Y))^{k}\right\}=z_i^k.$$

Applying Corollary \ref{cor:momentOfSum} together with (\ref{eqn:muiMu}) and (\ref{eqn:varSumVars}) yields
\begin{eqnarray*}
\E\left[\left|f(Y)-\E\left[ f(Y)\right]\right|^k\right]&\le& \E\left[\left(\sum_{i=1}^{m}\left|g^{(i)}(X)\right|\right)^k\right]\le  \left(\sum_{i=1}^{m}z_i\right)^k\le m^k \max_i z_i^k \\
&\le&   \max \left\{ \left(\sqrt{kR_4^{q} Var[f(Y)]}\right)^{k}, \max_{t\in [q]}(k^{t} R_4^q L^t \mu_t)^{k}\right\}
\end{eqnarray*}
where we choose $R_4\ge 1$ such that $m^{2k}R_3^{qk}2^{qk}\le R_4^{qk}$.
\end{proof}

\begin{proof}[Proof of Lemma \ref{lem:varBound}]
As in the proof of Lemma \ref{generalmoment} we write $f(Y) = \Eb{f(Y)} + \sum_i g^{(i)}(X)$ where $X_v = Y_v - \Eb{Y_v}$. Let ${\cal H}'$, $m$, $g^{(i)}$, $w^{(i)}$ and $q_i$ be defined as in that proof.
Using Lemma \ref{variance}, the inequality  (\ref{eqn:muiMu}) with $r=q_i$ and the central moment boundness we get
\begin{align}
Var[g^{(i)}(X)] & = \sum_{h\in {\cal  H}'} (w^{(i)}_{h})^2 \prod_{v \in \ve{h}} \Eb{X_v^2} \nonumber\\
& \le \sum_{h\in {\cal  H}'} \mu_{q_i}(w^{(i)}, Y)\cdot |w^{(i)}_{h}| \cdot \prod_{v \in \ve{h}} (2L\Eb{|X_v|}) \nonumber\\
& \le \sum_{h\in {\cal  H}'} \mu_{q_i}(w^{(i)}, Y)\cdot |w^{(i)}_{h}|\cdot \prod_{v \in \ve{h}} (4L\Eb{|Y_v|}) \nonumber \\
& = (4L)^{q_i} \mu_{q_i}(w^{(i)}, Y) \mu_0(w^{(i)}, Y)  \label{eqn:variance}
.\end{align}

Combining  Lemma \ref{orthogonal} and the inequality   (\ref{eqn:variance}) we get
\begin{align*}
Var[f(Y)] & = \sum_{i=1}^m Var[g^{(i)}(X)] \\
& \le 2q \max_{r\in [q]} (4L)^r \mu_r(w^{(i)}, Y) \mu_0(w^{(i)}, Y) \\
& \le 2q4^q\max_{r \in [q]} 4^rL^r \mu_r(w, Y) \mu_0(w, Y)
\end{align*}
where the last inequality uses (\ref{eqn:muiMu}).
\end{proof}

\section{Proof of the Theorem \ref{main}}\label{ProofMain}

Now we prove Theorem \ref{main} by applying Markov's inequality.
\begin{proof}
By Markov's inequality we derive
$$Pr[|f(Y)-\E\left[ f(Y)\right]|\ge \lambda]=Pr[|f(Y)-\E\left[ f(Y)\right]|^k\ge \lambda^k] \le \frac{\E[|f(Y)-\E\left[ f(Y)\right]|^k]}{\lambda^k}.$$
Choosing $k^*\ge 0$ to be the  even integer  such that $k^*\in (K-2,K]$ for
$$K=\min \left\{ \frac{\lambda^2}{e^2 R_4^q  Var[f(Y)] },\min_{t\in [q]}\left(\frac{\lambda}{ e R_4^q L^t \mu_t}\right)^{1/t}\right\}$$
i.e.\
\[
\frac{\sqrt{k^*R_4^q Var[f(Y)]}}{\lambda} \le 1/e \text{ and } \frac{(k^*)^{t} R_4^q L^t  \mu_t}{\lambda} \le 1/e
\]
for all $t\in [q]$. Using inequality (\ref{inequality4gen}) from Lemma \ref{generalmoment} we derive
\begin{eqnarray*}
Pr[|f(Y)-\E\left[ f(Y)\right]|\ge \lambda] &\le& \frac{\E\left[|f(Y)-\E\left[ f(Y)\right]|^{k^*}\right]}{\lambda^{k^*}} \\
& \le &
\max\left\{  e^{k^*\ln \frac{\sqrt{k^*R_4^q   Var[f(Y)]}}{\lambda}}, \max_{t\in [q]} e^{k^*\ln  \frac{(k^*)^{t} R_4^q L^t \mu_t}{\lambda}   }\right\}\\
&\le& e^{-k^*}\le  e^{-K+2}\nonumber\\
&\le& e^2 \cdot \max\left\{ e^{-\frac{\lambda^2}{R^q  Var[f(Y)]}},\max_{t\in [q]}e^{- \left(\frac{\lambda}{R^q L^t \mu_t}\right)^{1/t}}\right\}, \label{eqn:centeredMoment}
\end{eqnarray*}
for some universal constant $R>R_4$. This implies the statement of the Theorem.
\end{proof}

\section{Examples of Central Moment Bounded Random Variables}\label{sec:exampleMom}



\subsection{Bounded Random variables}

\begin{lemma}\label{lem:bmb}
Any random variable $Z$ with $|Z-\Eb{Z}| \le L$ is central moment bounded with parameter $L$.
\end{lemma}
\begin{proof}
For any $i \ge 1$ we clearly have $|Z-\Eb{Z}|^i \le L |Z-\Eb{Z}|^{i-1}$ hence $\Eb{|Z-\Eb{Z}|^i} \le L \Eb{|Z-\Eb{Z}|^{i-1}} \le i L \Eb{|Z-\Eb{Z}|^{i-1}}$.
\end{proof}

\subsection{Continuous log-concave random variables}

We say that non-negative function $f(x)$ is \emph{log-concave} if $f(\lambda x + (1-\lambda) y) \ge f(x)^\lambda f(y)^{1-\lambda}$ for any $0 \le \lambda \le 1$ and $x,y \in \R$ (see \cite{BV} Section 3.5). Equivalently $f$ is log concave if $\ln f(x)$ is concave on the set $\{x : f(x)>0\}$ where $\ln f(x)$ is defined and this set is a convex set (i.e.\ an interval).
A continuous random variable (or a continuous distribution) with density $f$ is \emph{log-concave} if $f$ is a log-concave function. See \cite{An,BB,BV} for introductions to log-concavity.

Schudy and Sviridenko \cite{SS} proved:
\begin{lemma}\label{lem:lcmb2} \cite{SS}
Any log-concave random variable $X$ with density $f$ is moment bounded with parameter $L=\frac{1}{\ln 2}\Eb{|X|} \approx 1.44 \Eb{|X|}$.
\end{lemma}

If $X$ is log-concave with density $f$ then $X - \Eb{X}$ clearly has density ${\tilde f}(x)=f(x+\Eb{X})$, which is evidently log-concave. Therefore:
\begin{corollary}\label{lem:lcmb2c}
Any log-concave random variable $X$ with density $f$ is central moment bounded with parameter $L=\frac{1}{\ln 2}\Eb{|X-\Eb{X}|} \approx 1.44 \Eb{|X-\Eb{X}|}\le 2.88\Eb{|X|}$.
\end{corollary}

\subsection{Discrete log-concave random variables}

A distribution over the integers $\dots, p_{-2},p_{-1},p_{0},p_{1},p_{2},\dots$ is said to be \emph{log-concave} \cite{An,JK} if $p_{i+1}^2 \ge p_i p_{i+2}$ for all $i$.
An integer-valued random variable $X$ is log-concave if its distribution $p_x = \pr{X=x}$ is.

The discrete case is a bit trickier than the continuous case since $X - \Eb{X}$ might take non-integer values even if $X$ takes integer ones. We therefore can only get inspiration from the proof in \cite{SS} that discrete log-concave random variables are moment bounded rather than using it as we did in the continuous case.

\begin{lemma}\label{lem:dlcmb2}
Let $X$ be a log-concave integer-valued random variable with $\pr{X \ge \ell} = 1$ and $\pr{X = \ell}>0$ for some $\ell \in \Z$.  Let   $a \in \R$ be an arbitrary real such that $a\le \ell \le a+1$. Then
\[
\Eb{|X - a|^k} \le k L \Eb{|X-a|^{k-1}}
\]
where $L= 1 + \Eb{|X-a|}$.
\end{lemma}
\begin{proof}
Let $u$ be the largest index $i$ such that $p_i > 0$ or infinity if there is no such index.  By log-concavity we have $p_i>0$ for all $i \in \Z$ with $\ell \le i \le u$. Let $r_i = \pr{X \ge i}$. Note that $r_\ell=1$ and when $u$ is finite $r_{i+1}=0$ for all $i \ge u$. We bound
\begin{align*}
\Eb{|X - a|^k}
 & = \sum_{x=\ell}^\infty (x-a)^k (r_x - r_{x+1}) \\
& = (\ell - a)^k + \sum_{x=\ell+1}^\infty r_x [(x-a)^k - (x-1-a)^k] \\
& \le (\ell - a)^k + \sum_{x=\ell+1}^\infty r_x k (x-a)^{k-1} \\
& \le \sum_{x=\ell}^\infty r_x k (x-a)^{k-1} \\
& = \sum_{x=\ell}^u \frac{r_x}{p_x} k (x-a)^{k-1} p_x \\
& \le \left(\sum_{x=\ell}^u \frac{r_x}{p_x} p_x\right)\left(\sum_{x=\ell}^u k (x-a)^{k-1} p_x\right) \\
& = \Eb{1+ |X-\ell|} \Eb{k|X-a|^{k-1}} \\
& \le (1 + \Eb{|X-a|}) k \Eb{|X-a|^{k-1}}
\end{align*}
where the second inequality uses the fact that $\ell - a \le 1 \le k$ and the third inequality follows from Chebychev's summation inequality, which applies because $r_x/p_x$ is a non-increasing sequence (Proposition 10 in \cite{An}) and $k (x-a)^{k-1}$ is a non-decreasing sequence.
\end{proof}

\begin{lemma}\label{lem:lcdmb2}
Any log-concave integer valued random variable $X$ such that $Pr\left[X=\Eb{X}\right]<1$ is central moment bounded with parameter $$L=1 + \max(\Eb{|X - \Eb{X}| ~|~ X \ge \Eb{X}}, \Eb{|X-\Eb{X}| ~|~ X < \Eb{X}}).$$
\end{lemma}
\begin{proof} It follows that $\pr{X \ge \Eb{X}}$ and $\pr{X < \Eb{X}}$ are both strictly positive, hence by log-concavity $\pr{X=\ceil{\Eb{X}}}$ and $\pr{X=\ceil{\Eb{X}} - 1}$ are both strictly positive.

Write
\begin{align*}
\Eb{|X-\Eb{X}|^k} &= \pr{X \ge \Eb{X}}\Eb{(X-\Eb{X})^k | X \ge \Eb{X}} + \pr{X < \Eb{X}}\Eb{(\Eb{X} - X)^k | X < \Eb{X}} \\
&= \pr{X \ge \Eb{X}}\Eb{(X_+ - \Eb{X})^k} + \pr{X < \Eb{X}}\Eb{(X_{-} + \Eb{X})^k}
\end{align*}
where $X_+$ is a random variable with $\pr{X_+=x} = \pr{X=x | X \ge \Eb{X}}$ and $X_-$ is a random variable with $\pr{X_-=x} = \pr{X=-x | X < \Eb{X}}$ for integer $x$. Clearly $X_+$ and $X_-$ inherit log-concavity from $X$. We also have $X_+ \ge \Eb{X}$ and $X_- \ge -\Eb{X}$.

We apply Lemma \ref{lem:dlcmb2} twice, first to $X_+$ with $a=\Eb{X}$ and $\ell = \ceil{\Eb{X}}$ and second to $X_-$ with $a=-\Eb{X}$ and $\ell=1-\ceil{\Eb{X}}$, yielding
\begin{align*}
\lefteqn{\Eb{|X-\Eb{X}|^k}} \\
 &= \pr{X \ge \Eb{X}}\Eb{(X_+ - \Eb{X})^k} + \pr{X < \Eb{X}}\Eb{(X_{-} + \Eb{X})^k} \\
 & \le \pr{X \ge \Eb{X}} k (1+ \Eb{|X_+ - \Eb{X}|}) \Eb{|X_+ - \Eb{X}|^{k-1}} + \\
 & \phantom{=} + \pr{X < \Eb{X}} k (1+\Eb{|X_- + \Eb{X}|}) \Eb{|X_- + \Eb{X}|^{k-1}} \\
&\le k (1+\max(\Eb{|X_+ - \Eb{X}|}, \Eb{|X_- + \Eb{X}|})) \cdot \\
& \phantom{=}\cdot \left(\pr{X \ge \Eb{X}}\Eb{|X - \Eb{X}|^{k-1} | X \ge 0} + \pr{X < \Eb{X}}\Eb{|X - \Eb{X}|^{k-1} | X < \Eb{X}}\right) \\
& = k (1+\max(\Eb{|X-\Eb{X}|~|~X \ge \Eb{X}}, \Eb{|X-\Eb{X}|~|~X < \Eb{X}})) \Eb{|X-\Eb{X}|^{k-1}}
.\end{align*}
\end{proof}



\begin{thebibliography}{2}


\bibitem{An} M.\ An, Log-concave Probability Distributions: Theory and Statistical Testing, in Game Theory and Information EconWPA 9611002 (1996).

\bibitem{BB} M.\ Bagnoli and T.\ Bergstrom, Log-concave probability and its applications, Economic Theory 26(2) (2005), pp.\ 445-469.


\bibitem{B1} S.N. Bernstein, Theory of Probability, (Russian), Moscow, 1927.

\bibitem{B} S.N. Bernstein,  On certain modifications of Chebyshev's inequality. Doklady Akademii Nauk SSSR 17 (6) (1937), pp. 275–-277.




\bibitem{BV} S. Boyd and L. Vandenberghe, Convex optimization, Cambridge University Press, 2004.

\bibitem{BCL} C. Buchheim, A. Caprara and A. Lodi, An Effective Branch-and-Bound Algorithm for Convex Quadratic Integer Programming, in Proceedings of IPCO 2010, pp. 285--298.

\bibitem{DHK} M. Dudik, D.  Hsu, S. Kale, N. Karampatziakis, J. Langford, L. Reyzin and T. Zhang, Efficient Optimal Learning for Contextual Bandits, in Proceedings of UAI 2011.


\bibitem{FR} D. Freedman,   On tail probabilities for martingales. Ann. Probability 3 (1975), pp. 100–-118.

\bibitem{jacod} J.\ Jacod and P.\ Protter. Probability Essentials, Springer-Verlag (2004).

\bibitem{J} S. Janson,    Gaussian Hilbert spaces. Cambridge Tracts in Mathematics. 129. Cambridge: Cambridge University Press, (1997).

\bibitem{JK} N. Johnson and S. Kotz, Discrete Distributions, John Wiley, New York (1969).

\bibitem{KV} J. Kim and V. Vu,   Concentration of multivariate polynomials and its applications, Combinatorica 20 (2000), no. 3, 417-434.

\bibitem{L} R. Latala, Estimates of moments and tails of Gaussian chaoses, Ann. Probab. 34 (2006), pp. 2315-2331.

\bibitem{MSS} K. Makarychev, W. Schudy and M. Sviridenko, Concentration Inequalities for Nonlinear Matroid Intersection, to appear in SODA2012.

\bibitem{mcdiarmid98} C. McDiarmid, Concentration, In Probabilistic Methods for Algorithmic Discrete Mathematics, M. Habib, C. McDiarmid, J. Ramirez-Alfonsin and B. Reed editors, pp. 195--248, Springer, 1998.

\bibitem{MOO} E. Mossel, R. O'Donnell, and K. Oleszkiewicz. Noise stability of functions with low influences: invariance and optimality. Annals of Mathematics 171(1), pp.\ 295-341 (2010).


\bibitem{RT} R. Prabhakar Raghavan and C. Thompson, Randomized rounding: a technique for provably good algorithms and algorithmic proofs. Combinatorica 7(4): 365-374 (1987).


\bibitem{SS} W. Schudy and M. Sviridenko, Concentration and Moment Inequalities for Polynomials of Independent Random Variables, submitted for publication, extended abstract appeared in Proceedings of SODA2012.



\bibitem{U}   J. V. Uspensky, Introduction to Mathematical Probability, McGraw-Hill Book Company, 1937.

\bibitem{V1} V. Vu,  Concentration of non-Lipschitz functions and applications, Probabilistic methods in combinatorial optimization. Random Structures Algorithms 20 (2002), no. 3, 262-316.


\bibitem{V2} V. Vu,   On the concentration of multivariate polynomials with small expectation, Random Structures Algorithms 16 (2000), no. 4, 344-363.


\end{thebibliography}
\end{document}